\documentclass[11pt, twoside]{article}

\usepackage[english]{babel}

\usepackage{amssymb}
\usepackage{mathrsfs}
\usepackage{amsmath}
\usepackage{amsthm}
\usepackage{amsfonts}
\usepackage{latexsym}
\usepackage{indentfirst}
\usepackage{color}
\usepackage{txfonts}
\usepackage{enumerate}

\usepackage[colorlinks=true,
linkcolor=blue,
citecolor=red,
urlcolor=magenta,
]{hyperref}

\usepackage{txfonts}
\usepackage{anysize}

\allowdisplaybreaks

\newtheorem{theorem}{Theorem}[section]
\newtheorem{lemma}[theorem]{Lemma}

\theoremstyle{definition}

\newtheorem{definition}[theorem]{Definition}
\newcounter{assum}

\renewcommand{\appendix}{\par
\setcounter{section}{0}%
\setcounter{subsection}{0}%
\setcounter{subsubsection}{0}%
\gdef\thesection{\@Alph\c@section}%
\gdef\thesubsection{\@Alph\c@section.\@arabic\c@subsection}%
\gdef\theHsection{\@Alph\c@section.}%
\gdef\theHsubsection{\@Alph\c@section.\@arabic\c@subsection}%
\csname appendixmore\endcsname
}

\numberwithin{equation}{section}

\begin{document}
\title{\bf\Large
A Counterexample to Kenig's Interpolation Problem for Sobolev
Spaces with Zero Boundary Conditions\footnotetext{\hspace{-0.35cm} 2020 {\it
Mathematics Subject Classification}.
Primary 46E35; Secondary 46B70, 35J25, 35B65, 35J67.
\endgraf {\it Key words and phrases.}
Sobolev space, complex interpolation, zero boundary condition,
Dirichlet problem, Lipschitz domain.
\endgraf This project is partially supported by the National
Natural Science Foundation of China (Grant Nos. 12431006, 12371093, and 12501118),
the Beijing Natural Science Foundation (Grant No. 1262011),
the Fundamental Research Funds for the Central Universities
(Grant No. 2253200028), and Longyuan Young Talents of Gansu Province.
}}
\author{Xiaosheng Lin, Dachun Yang\footnote{Corresponding
author, E-mail: \texttt{dcyang@bnu.edu.cn}/{\color{red}\today}/Final version.},
\ \ Sibei Yang, Wen Yuan
and Yangyang Zhang}
\date{}
\maketitle

\vspace{-0.7cm}

\begin{center}
\begin{minipage}{13cm}
{\small {\bf Abstract.}\quad
Let $n\in \mathbb N\cap[2,\infty)$. In this article, we show that there exists a
bounded $C^1$ domain $\Omega\subset \mathbb R^n$ such that, for any given
$s\in(1,2)\setminus\{\frac32\}$,
\begin{align*}
\left[H_0^1(\Omega),H^2(\Omega)\cap H_0^1(\Omega)\right]_{s-1}
=H^s(\Omega)\cap H_0^1(\Omega)=H_0^s(\Omega)
\end{align*}
with equivalent norms, but
\begin{align*}
\left[H_0^1(\Omega),H^2(\Omega)\cap H_0^1(\Omega)\right]_{\frac12}
\subsetneqq H^{\frac32}(\Omega)\cap H_0^1(\Omega),
\end{align*}
which provides a counterexample to Problem 3.3.19 of Kenig in
[CBMS Regional Conf. Ser. in Math. 83, 1994].
As applications, we prove that for such a domain $\Omega$
\begin{align*}
H^2(\Omega)\cap H_0^1(\Omega)\subsetneqq D(-\Delta_D)
\end{align*}
(the domain of the Dirichlet Laplacian operator $-\Delta_D$ on $\Omega$)
and construct a solution of the homogeneous heat equation
with zero Dirichlet boundary condition, which does not belong to
$L^2((0,T);H^2(\Omega)\cap H_0^1(\Omega))$
for any given $T\in(0,\infty)$.
}
\end{minipage}
\end{center}

\vspace{0.1cm}

%\tableofcontents

\section{Introduction}
The interpolation theory of function spaces is essentially a bridge connecting
function spaces of different scales.
This is precisely why the properties of the endpoint spaces play a decisive
role in determining the properties of the interpolation spaces.
This theory is not only an abstract tool in functional analysis
but also a central pillar in fields, such as harmonic analysis
and partial differential equations (see, for example,
\cite{bl76,lm72,lp61,lp64,mm14,t78}). In particular,
the interpolation theory of Sobolev spaces
on domains has important and extensive applications in partial differential equations
(see, for example, \cite{chm15,jk95,lm72,mm14}).

The goal of this article is to answer a question raised by Kenig \cite[Problem 3.3.19]{k94}
concerning the complex interpolation of Sobolev spaces on
bounded Lipschitz domains in $\mathbb{R}^n$ with zero boundary conditions.
To describe this question, we need to recall some necessary concepts.

Denote by $\mathcal{S}(\mathbb{R}^n)$ the \emph{space of all Schwartz functions} on $\mathbb{R}^n$ equipped with
the well-known topology determined by a countable family of norms and by $\mathcal{S}'(\mathbb{R}^n)$
its \emph{dual space} (that is, the space of all tempered distributions
equipped with the weak-$\ast$ topology).
Let $O\subset\mathbb{R}^n$ be an open set. Denote by
the \emph{symbol $C_{\rm c}^\infty(O)$} the set of all smooth functions on $O$
with compact support contained in $O$ equipped with inductive
limit topology. Write ${\mathcal D}(O):=C_{\rm c}^\infty(O)$ and  denote by ${\mathcal D'}(O)$
the space of all continuous linear functionals on ${\mathcal D}(O)$
equipped with the weak-$\ast$ topology.

Let $n\ge2$, $s\in\mathbb{R}$, and $\Omega\subset\mathbb{R}^n$ be a domain,
which means it is a connected open set. Throughout the article, the
\emph{Bessel potential Sobolev space} $H^s(\mathbb R^n)$ is defined as the set
of all tempered distributions $u\in\mathcal{S}'(\mathbb{R}^n)$ satisfying
\begin{align*}
\|u\|_{H^s(\mathbb R^n)}:=\left[\int_{\mathbb{R}^n}(1+|\xi|^2)^{s}|\widehat{u}(\xi)|^2\,d\xi\right]^{\frac{1}{2}}<\infty,
\end{align*}
where $\widehat{u}$ denotes the Fourier transform of $u$.
Recall that, for any $u\in \mathcal{S}(\mathbb{R}^n)$,
its Fourier transform $\widehat{u}$ is defined by setting,
for any $\xi\in \mathbb{R}^n$,
$$\widehat{u}(\xi):=\int_{\mathbb{R}^n}u(x)e^{-2\pi ix\xi}\,dx.$$
Moreover, the \emph{restricted Sobolev space} $H^s(\Omega)$
is defined by setting
\begin{align*}
H^s(\Omega):=\{f\in \mathcal{D}'(\Omega):\ f=g|_\Omega\text{ for some }g\in H^s(\mathbb{R}^n)\}.
\end{align*}
Meanwhile, for any $f\in H^s(\Omega)$, let
\begin{align*}
\left\|f\right\|_{H^s(\Omega)}:=\inf\left\{\|g\|_{H^s(\mathbb{R}^n)}:\
f=g|_\Omega,\ g\in H^s(\mathbb{R}^n)\right\}.
\end{align*}
Moreover, define
\begin{align*}
H_0^s(\Omega):=\overline{C_{\rm c}^\infty(\Omega)}^{H^s(\Omega)}.
\end{align*}
In \cite[Problem 3.3.19]{k94}, Kenig asked whether, for any given bounded Lipschitz domain $\Omega,$
\begin{align}\label{nianqing}
\left[H_0^1(\Omega),H^2(\Omega)\cap H^1_0(\Omega)\right]_{\frac{1}{2}}
=H^{\frac{3}{2}}(\Omega)
\cap H^1_0(\Omega),
\end{align}
where $[H_0^1(\Omega),H^2(\Omega)\cap H^1_0(\Omega)]_{\frac{1}{2}}$
denotes the \emph{complex interpolation space} between $H_0^1(\Omega)$ and
$H^2(\Omega)\cap H_0^1(\Omega)$ with interpolation parameter $1/2.$

In this article, we construct a counterexample to show that
\eqref{nianqing} does not hold in general.
More precisely, we have the following theorem.
\begin{theorem}\label{thm:main}
Let $n\in\mathbb{N}\cap[2,\infty).$
Then there exists a bounded $C^1$ domain $\Omega\subset\mathbb{R}^n$
such that
\begin{align}\label{keypoint}
\left[H_0^1(\Omega),H^2(\Omega)\cap H^1_0(\Omega)\right]_{\frac{1}{2}}
\subsetneqq H^{\frac{3}{2}}(\Omega)\cap H^1_0(\Omega).
\end{align}
\end{theorem}

This theorem has the following four applications.

\textbf{(I) Criticality of the interpolation scale.} Theorem \ref{thm:main} indicates that the index $s=\frac{3}{2}$ is a distinguished point
in the complex interpolation scale with zero boundary conditions. Indeed, the
expected interpolation identity may fail exactly at
$s=\frac{3}{2}$. The next theorem shows that this phenomenon is genuinely
critical: for any given bounded Lipschitz domain, the interpolation formula holds
throughout the subcritical range $s\in(1,\frac{3}{2})$, while, for the
particular bounded $C^1$ domain in Theorem \ref{thm:main}, it also remains valid for all
$s\in(\frac{3}{2},2)$. Thus, the obstruction is concentrated at the single
critical index $s=\frac{3}{2}$.
\begin{theorem}\label{thm:sharp}
Let $n\in\mathbb{N}\cap[2,\infty)$.
\begin{itemize}
\item  [$\mathrm{(i)}$]
For any given bounded Lipschitz domain $\Omega$ and for any $s\in (1,\frac{3}{2})$,
\begin{align}\label{yumao}
\left[H_0^1(\Omega),H^2(\Omega)\cap H^1_0(\Omega)\right]_{s-1}
=H^{s}(\Omega)\cap H^1_0(\Omega)=H^{s}_0(\Omega)
\end{align}
with equivalent norms.

\item[$\mathrm{(ii)}$]
If $\Omega$ is the same as in
Theorem \ref{thm:main}, then, for any $s\in(1,2)\setminus\{\frac{3}{2}\}$, \eqref{yumao}
holds with equivalent norms.
\end{itemize}
\end{theorem}
\textbf{(II) Optimality of the endpoint solution space for the Dirichlet problem.}
 Consider the following inhomogenous Dirichlet problem on the bounded Lipschitz domain $\Omega$:
\begin{align}\label{eq1.1}
\begin{cases}
\Delta v=f&\ \text{in}\ \ \Omega,\\
v=0&\ \text{on}\ \ \partial\Omega,
\end{cases}
\end{align}
where $\partial\Omega$ denotes the boundary of $\Omega$.
By \cite[Theorem B]{jk95}, we find that, for any given $f\in L^2(\Omega)$, the weak solution $v$ of \eqref{eq1.1}
satisfies $v\in H^{\frac{3}{2}}(\Omega)\cap H_0^1(\Omega)$. From \eqref{keypoint} and the proof of Theorem \ref{thm:main},
we deduce that $H^{\frac{3}{2}}(\Omega)\cap H_0^1(\Omega)$ cannot be characterized by $[H_0^1(\Omega),H^2(\Omega)
\cap H^1_0(\Omega)]_{\frac{1}{2}}$ and $v\not\in [H_0^1(\Omega),H^2(\Omega)
\cap H^1_0(\Omega)]_{\frac{1}{2}}$ in general. Furthermore, by \cite[Corollary 3.2]{c19}, we conclude that
there exists a bounded $C^1$ domain $\Omega$ such that, for any $f\in C^\infty(\overline{\Omega})$, the weak solution
$v$ of the Dirichlet problem \eqref{eq1.1} does not belong to $H^{\frac32+\varepsilon}(\Omega)$ for any $\varepsilon\in(0,\infty)$.
Thus, for any given bounded Lipschitz domain $\Omega$, the solution
space $H^{\frac{3}{2}}(\Omega)\cap H_0^1(\Omega)$ for
the Dirichlet problem \eqref{eq1.1} with $f\in L^2(\Omega)$ is \emph{optimal}
in the sense that it cannot be improved to either
the interpolation space $[H_0^1(\Omega),H^2(\Omega)\cap H^1_0(\Omega)]_{\frac{1}{2}}$
or the Sobolev space
$H^{\frac32+\varepsilon}(\Omega)\cap H^1_0(\Omega)$ for any $\varepsilon\in(0,\infty)$.

\textbf{(III) The domain of the Dirichlet Laplacian.}
Let $\Omega$ be a
bounded $C^2$ domain and let
$\Delta_D$ be the Dirichlet Laplacian on $\Omega$.
By \cite[Section 6.3.2]{e10}, we find that
\begin{align*}
 D(-\Delta_D)=H^2(\Omega)\cap H_0^1(\Omega),
\end{align*}
where $D(-\Delta_D)$ denotes the domain of $-\Delta_D.$
However, Theorem \ref{thm:Laplacian-application} below shows that this characterization  fails on
bounded $C^1$ domains in general, which is essentially contained in the proof
of Theorem \ref{thm:main}.

\begin{theorem}\label{thm:Laplacian-application}
Let $\Omega$ be the same as in Theorem \ref{thm:main}. Then
\begin{align*}
 H^2(\Omega)\cap H_0^1(\Omega)
 \subsetneqq
 D(-\Delta_D).
\end{align*}
More precisely, let $v$ be the weak solution of \eqref{eq1.1} with $f\equiv 1$
and, for any $t\in[0,\infty)$, let $ u(t):=e^{t\Delta_D}v.$
Then
\begin{align*}
 u(t)\in D(-\Delta_D)
 \quad\text{but}\quad
 u(t)\notin H^2(\Omega)\cap H_0^1(\Omega).
\end{align*}
\end{theorem}
\textbf{(IV) A heat equation without the expected $H^2$-regularity.}
Let  $\Omega$ be a
bounded $C^\infty$ domain
and $v,u$ be as in
Theorem \ref{thm:Laplacian-application}. Then $u$ is the unique solution of
the following homogeneous heat equation with Dirichlet boundary condition
\begin{equation}\label{eq:heat}
 \begin{cases}
  \partial_t u-\Delta u=0 & \text{in } \Omega\times(0,T),\\
  u=0 & \text{on } \partial\Omega\times(0,T),\\
  u(0)=v & \text{in } \Omega.
 \end{cases}
\end{equation}
Since $\Omega$ is $C^\infty$,
by the standard regularity theory for the heat equation (see, for example,
\cite[Section 7.1.3]{e10}),  it follows that
\begin{align*}
u\in L^2\big((0,T);\,H^2(\Omega)\cap H_0^1(\Omega)\big),
\end{align*}
where $L^2\big((0,T);H^2(\Omega)\cap H_0^1(\Omega)\big)$ denotes the Bochner space of square-integrable functions taking values in the space
$H^2(\Omega)\cap H_0^1(\Omega).$
By Theorem \ref{thm:Laplacian-application},
we immediately obtain the following conclusion,
which shows that such a conclusion
may fail on bounded $C^1$ domains; we omit the details here.
\begin{theorem}\label{thm:heat-application}
Let $\Omega,v, $ and $u$ be the same as in Theorem
\ref{thm:Laplacian-application}. Then, for any given $T\in(0,\infty)$, $u$ is the weak solution
of \eqref{eq:heat}, but
\begin{align*}
 u\notin L^2\big((0,T);\,H^2(\Omega)\cap H_0^1(\Omega)\big).
\end{align*}
\end{theorem}

The proofs of Theorems \ref{thm:main}, \ref{thm:sharp},
and \ref{thm:Laplacian-application}  are  given in Section \ref{diyideng}.

We end this introduction by making some conventions on the notation. Let
${\mathbb N}:=\{1,2,\ldots\}$. For any $E\subset\mathbb{R}^n$,
we use $\mathbf{1}_E$ to denote its
characteristic function and use $\overline{E}$ to denote its closure
in $\mathbb{R}^n$.
For any given Banach space $X$ and $Y$, the notation
$X\hookrightarrow Y$ means that
$X$ is continuously embedded in
$Y$.
Finally, in all proofs we consistently retain
the notation introduced in the original theorem (or related statement).

\section{Proofs of Main Results}\label{diyideng}

In this section, we prove
Theorems \ref{thm:main}, \ref{thm:sharp}, and  \ref{thm:Laplacian-application}.
We  begin by recalling some necessary concepts and facts on real
and complex interpolations for Banach spaces. Suppose that $X_0$ and
$X_1$ are two complex Banach spaces. The couple $(X_0, X_1)$ is
said to be \emph{compatible} if $X_0$ and $X_1$ are continuously embedded
into a common Hausdorff topological vector space $X$. In this case, we
can naturally define two Banach spaces $X_0+X_1$ and $X_0 \cap X_1$.
More precisely, let
\begin{align*}
X_0+X_1:=\left\{x\in X:x=x_0+x_1,\ x_0\in X_0,\mbox{\ and\ }x_1\in X_1\right\}
\end{align*}
equipped with the norm
\begin{align*}
\|x\|_{X_0+X_1}:=\inf\{\|x_0\|_{X_0}+\|x_1\|_{X_1}:
x_0\in X_0,\ x_1\in X_1,\mbox{\ and\ }x=x_0+x_1\}
\end{align*}
and, for any $x\in X_0\cap X_1$, let
$\|x\|_{X_0 \cap X_1}:=\max\{\|x\|_{X_0},\,\|x\|_{X_1}\}$.
We refer to \cite[Lemma 2.3.1]{bl76} for some basic properties 
of $X_0+X_1$ and $X_0 \cap X_1$.

\begin{definition}
Let $(H_0,H_1)$ be a compatible couple of Hilbert spaces. For any $t\in(0,\infty)$ and $h\in H_0+H_1$, the
\emph{$K$-functional} $K(t,h;H_0,H_1)$ is defined by setting
$$
K(t,h;H_0,H_1):=\inf\left\{\|h_0\|_{H_0}+t\|h_1\|_{H_1}:h=h_0+h_1,\ h_0\in H_0,\ h_1\in H_1\right\}.
$$
Let $\theta\in(0,1)$. The \emph{Hilbert real interpolation space $(H_0,H_1)_{\theta,2}$ of type $(\theta,2)$}
is defined as the set of all $h\in H_0+H_1$ such that
$$
\|h\|_{(H_0,H_1)_{\theta,2}}:=\left\{\int_0^\infty\left[t^{-\theta}
K(t,h;H_0,H_1)\right]^2\,\frac{dt}{t}\right\}^{\frac12}
<\infty.
$$
\end{definition}

Next, we recall the concept of Calder\'on's complex interpolation spaces (see \cite[p.\,114, 3]{c64}
or \cite[p.\,88]{bl76}). To this end, let $S:= \{ z\in\mathbb{C}:0<\Re(z)<1\}$ and $\overline{S}$ be its closure in
$\mathbb{C}$, where $\Re(z)$ denotes the \emph{real part} of $z$.

\begin{definition}
Let $(X_0, X_1)$ be a compatible couple of complex Banach spaces.
\begin{enumerate}
\item[\rm(i)] The \emph{space $\mathcal{F}(X_0, X_1)$} is defined to be the
set of all functions $F:\overline{S}\to X_0+X_1$ such that
\begin{itemize}
\item [$\mathrm{(i)_1}$] $F$ is bounded and continuous on $\overline{S}$,
\item [$\mathrm{(i)_2}$] $F$ is analytic in $S$,
\item [$\mathrm{(i)_3}$] for any $j\in\{0,\,1\}$, the function $t\in\mathbb{R}
\longmapsto F(j+it)\in X_j$ is bounded and continuous.
\end{itemize}
Moreover, the space $\mathcal{F}(X_0,X_1)$ is equipped with the  following norm that,
for any $F\in\mathcal{F}(X_0,X_1)$,
\begin{align*}
\|F\|_{\mathcal{F}(X_0,X_1)}:=\max\left\{\sup_{z\in
i\mathbb{R}}\|F(z)\|_{X_0},\,\sup_{z\in 1+i\mathbb{R}}\|F(z)\|_{X_1}\right\}.
\end{align*}
\item[\rm(ii)] Let $\theta\in(0,1)$. The \emph{complex interpolation
space $[X_0,X_1]_\theta$} with respect to $(X_0, X_1)$ is defined to be the
set of all functions $f\in X_0+X_1$ such that $f=F(\theta)$ for some
$F\in\mathcal{F}(X_0,X_1)$, equipped with the norm
\begin{align*}
\|f\|_{[X_0,X_1]_\theta}:=\inf\left\{\|F\|_{\mathcal{F}(X_0,X_1)}:
f=F(\theta)\mbox{\ for\ some\ }F\in\mathcal{F}(X_0,X_1)\right\}.
\end{align*}
\end{enumerate}
\end{definition}
As a part of  \cite[Theorem 4.1.2]{bl76}, we have the following result.
\begin{lemma}
\label{lem:functoriality}
Let $(X_0,X_1)$ and $(Y_0,Y_1)$ be compatible couples of Banach spaces and
$T:X_0+X_1\to Y_0+Y_1$ be linear. Assume  that
$T:\ X_0\to Y_0$ and $T:\ X_1\to Y_1$ are bounded linear with operator norms at most $M_0$ and $M_1$ respectively. Then, for any given  $\theta\in (0,1)$,
$T:\ [X_0,X_1]_{\theta}\to [Y_0,Y_1]_{\theta}$
is bounded with operator norm at most $M_0^{1-\theta}M_1^\theta$.
In particular, if $X_j\hookrightarrow Y_j$ continuously for $j=0,1$, then
$[X_0,X_1]_{\theta}\hookrightarrow [Y_0,Y_1]_{\theta}.$
\end{lemma}
Let $s\in \mathbb{R},$ $\Omega\subset\mathbb{R}^n$ be a domain, and $F\subset \mathbb{R}^n$ be closed set. Recall that
the Sobolev spaces $\widetilde{H}^s(\Omega)$ and $H^s_F(\mathbb{R}^n)$ are defined, respectively, by setting
\begin{align*}
\widetilde{H}^s(\Omega):=\overline{C_{\rm c}^\infty(\Omega)}^{H^s(\mathbb{R}^n)}
\end{align*}
and
\begin{align*}
H^s_F(\mathbb{R}^n):=\left\{f\in H^s(\mathbb{R}^n):\ \mathrm{supp}\,(f)\subset F\right\}.
\end{align*}
Moreover, the restriction map $R_\Omega:\widetilde{H}^s(\Omega)\to \mathcal{D}'(\Omega)$ is defined by setting,
for any $f\in\widetilde{H}^s(\Omega),$  $R_\Omega f:=f|_{\Omega}.$
The following lemma is precisely \cite[Theorems 3.29 and 3.33]{m00}.
\begin{lemma}\label{dengtong}
\begin{itemize}
\item  [$\mathrm{(i)}$]
Let $\Omega$ be a bounded $C^0$ domain in $\mathbb{R}^{n}$. Then, for any $s\in\mathbb{R},$
\begin{align*}
\widetilde{H}^s(\Omega)=H^s_{\overline{\Omega}}(\mathbb{R}^n).
\end{align*}
\item  [$\mathrm{(ii)}$]  Let $\Omega$ be a bounded Lipschitz domain in $\mathbb{R}^{n}$. Then, for any $s\in(0,\infty)$ with $s\not\in \{\frac{1}{2},\frac{3}{2},\frac{5}{2},\ldots\},$
\begin{align*}
\widetilde{H}^s(\Omega)=H_0^s(\Omega).
\end{align*}
\end{itemize}
\end{lemma}
From \cite[Theorem 2.9]{m11} and \cite[(1.10)]{c19}, we deduce the following conclusions.
\begin{lemma}\label{lem:Costabel}
\begin{itemize}
\item [$\mathrm{(i)}$] Let $\Omega$ be a Lipschitz domain in $\mathbb{R}^{n}$. Then, for any $s\in (1,\frac{3}{2}),$
\begin{align*}
H^s_0(\Omega)=\left\{u\in H^s(\Omega):\ \gamma(u)=0\ \text{on}\ \partial\Omega\right\}
=H^s(\Omega)\cap H^1_0(\Omega),
\end{align*}
where $\gamma$ denotes the trace operator on $\partial\Omega$.
\item [$\mathrm{(ii)}$] Let $n\in\mathbb{N}\cap [2,\infty)$. Then there exists a bounded $C^1$ domain $\Omega\subset\mathbb{R}^n$
such that,
for any $s\in (\frac{3}{2},\frac{5}{2})$,
\begin{align*}
H^{s}(\Omega)\cap H^1_0(\Omega)
=H^{s}_0(\Omega).
\end{align*}
\end{itemize}
\end{lemma}
The following lemma is a  part of
\cite[Corollaries 4.7 and 4.10 and Remark 3.6]{chm15}.
\begin{lemma}\label{shichazhi}
Let $\Omega$ be a bounded Lipschitz domain in $\mathbb{R}^{n}$. Then, for any $s_0,s_1\in\mathbb{R}$ and $\theta\in (0,1)$,
\begin{align*}
\left[\widetilde H^{s_0}(\Omega),\widetilde H^{s_1}(\Omega)\right]_\theta
=\left(\widetilde H^{s_0}(\Omega),\widetilde H^{s_1}(\Omega)\right)_{\theta,2}
=
\widetilde H^{(1-\theta)s_0+\theta s_1}(\Omega)
\end{align*}
and
\begin{align*}
\left[H^{s_0}(\Omega), H^{s_1}(\Omega)\right]_\theta
=
\left( H^{s_0}(\Omega),H^{s_1}(\Omega)\right)_{\theta,2}
=
H^{(1-\theta)s_0+\theta s_1}(\Omega)
\end{align*}
with equivalent norms.
\end{lemma}
As a part of \cite[Theorem 2.10]{m11}, we have the following conclusion.
\begin{lemma}\label{zuishen}
Let $\Omega$ be a bounded Lipschitz domain in $\mathbb{R}^{n}$ and $s\in
 [-\frac{1}{2},\infty)$. Then $H^s_{\partial \Omega}(\mathbb{R}^{n})=\{0\}.$
\end{lemma}
Now, we prove Theorem \ref{thm:main}.
\begin{proof}[Proof of Theorem \ref{thm:main}]
Let $\Omega\subset\mathbb{R}^n$ be a bounded $C^1$ domain as in Lemma \ref{lem:Costabel}(ii) and
let $v\in H^1_0(\Omega)$ be the weak solution of the Dirichlet problem
\begin{align}\label{pos}
\begin{cases}
\Delta v(x)=1&\ \text{in}\ \ \Omega,\\
v=0&\ \text{on}\ \ \partial\Omega.
\end{cases}
\end{align}
That is, for any $\phi\in C^\infty_{\mathrm{c}}(\Omega),$
\begin{align}\label{feng}
\int_{\mathbb{R}^n}\nabla v(x)\cdot \nabla\phi(x)\,dx=-\int_{\mathbb{R}^n}\phi(x)\,dx.
\end{align}
By \cite[Theorem B]{jk95}, we conclude that $v\in H^{\frac{3}{2}}(\Omega)\cap H_0^1(\Omega)$.
From Lemma \ref{lem:Costabel}(ii),
we infer that
\begin{align}\label{e3.2}
\left[H_0^1(\Omega),H^2(\Omega)\cap H^1_0(\Omega)\right]_{\frac{1}{2}}
=\left[H_0^1(\Omega),H^2_0(\Omega)\right]_{\frac{1}{2}}.
\end{align}
Meanwhile, it is known that
$$\left[H_0^1(\Omega),H^2_0(\Omega)\right]_{\frac{1}{2}}
=R_\Omega\left(\widetilde{H}^{\frac{3}{2}}(\Omega)\right)$$
(see, for instance, \cite[Remark 3.32]{chm17}), which, combined with
\eqref{e3.2}, implies that
$$\left[H_0^1(\Omega),H^2(\Omega)\cap H^1_0(\Omega)\right]_{\frac{1}{2}}
=R_\Omega\left(\widetilde{H}^{\frac{3}{2}}(\Omega)\right).
$$
Thus, to prove \eqref{keypoint}, we only need to show that $v\not\in R_\Omega(\widetilde{H}^{\frac{3}{2}}(\Omega)).$
Assume for contradiction that
$v\in R_\Omega(\widetilde{H}^{\frac{3}{2}}(\Omega)).$
Then there exists $U\in \widetilde{H}^{\frac{3}{2}}(\Omega)
$
such that
$R_\Omega U=v.$
Let
\begin{align*}
\mu:=\Delta U-\mathbf 1_\Omega.
\end{align*}
Since $U\in \widetilde{H}^{\frac{3}{2}}(\Omega)\subset H^{\frac{3}{2}}(\mathbb{R}^n),$
it follows that $\Delta U\in H^{-\frac{1}{2}}(\mathbb{R}^n)$.  Observe that $\mathbf{1}_{\Omega}
\in L^2(\mathbb{R}^n)\subset H^{-\frac{1}{2}}(\mathbb{R}^n).$
Thus, $\mu\in H^{-\frac{1}{2}}(\mathbb{R}^n).$

By the definition of $\mu$ and Lemma \ref{dengtong}(i), we find that
\begin{align}\label{caomei}
\mathrm{supp}\,(\mu)\subset \overline{\Omega}.
\end{align}
Moreover, from \eqref{feng}, we deduce that, for any $\phi\in C_{\rm c}^\infty(\Omega),$
\begin{align*}
\langle \Delta v,\phi\rangle
=-\int_{\mathbb{R}^n}\nabla v(x)\cdot \nabla \phi(x)\,dx
=\int_{\mathbb{R}^n}\phi(x)\,dx
\end{align*}
and hence
\begin{align*}
\langle\mu,\phi\rangle
=\langle\Delta  U,\phi\rangle-\int_\Omega \phi(x)\,dx
=\langle\Delta  v,\phi\rangle-\int_\Omega \phi(x)\,dx=0.
\end{align*}
This, together with \eqref{caomei}, implies that
$\mathrm{supp}\,(\mu)\subset\partial\Omega$.
By this and $\mu\in H^{-\frac{1}{2}}(\mathbb{R}^n)$, we find that $\mu\in H^{-\frac{1}{2}}_{\partial \Omega}(\mathbb{R}^n),$
which, combined with Lemma \ref{zuishen}, further implies that $\mu=0.$
Thus,
\begin{align}\label{niua}
\Delta U=\mathbf 1_\Omega
\ \ \text{in}\ \ \mathcal D'(\mathbb{R}^n).
\end{align}
Let $\varphi\in C^\infty_{\mathrm{c}}(\mathbb{R}^n)$  satisfy  $\varphi\equiv1$
in $\overline{\Omega}.$
Since $\mathrm{supp}\,(U)\subset \overline{\Omega}$ and $\Delta \varphi\equiv0$ in $\overline{\Omega}$, it follows that
\begin{align}\label{niua2}
\langle \Delta U,\varphi\rangle
=\langle  U,\Delta\varphi\rangle
=0.
\end{align}
On the other hand,
\begin{align*}
\langle \mathbf{1}_{\Omega},\varphi\rangle
=\int_{\Omega}\varphi(x)\,dx
=|\Omega|.
\end{align*}
From this, \eqref{niua}, and \eqref{niua2}, we infer that $|\Omega|=0,$ which is impossible.
Thus, $v\notin R_\Omega(\widetilde{H}^{\frac{3}{2}}(\Omega))$, which 
completes the proof of Theorem \ref{thm:main}.
\end{proof}

\begin{proof}[Proof of Theorem \ref{thm:sharp}]
We first show (i).
Let $\Omega$ be a bounded Lipschitz domain in $\mathbb{R}^{n}$ and $s\in (1,\frac{3}{2}).$
By the fact that $H^1_0(\Omega)\hookrightarrow H^1(\Omega)$ and $H^2(\Omega)\cap H^1_0(\Omega)\hookrightarrow H^2(\Omega)$
and by
Lemmas \ref{lem:functoriality} and \ref{shichazhi},
we conclude that
\begin{align*}
\left[H_0^1(\Omega),H^2(\Omega)\cap H^1_0(\Omega)\right]_{s-1}
\hookrightarrow
\left[H^1(\Omega),H^2(\Omega)\right]_{s-1}=H^s(\Omega).
\end{align*}
On the other hand, using Lemma \ref{lem:functoriality} again, we obtain
\begin{align*}
\left[H_0^1(\Omega),H^2(\Omega)\cap H^1_0(\Omega)\right]_{s-1}
\hookrightarrow
\left[H_0^1(\Omega),H_0^1(\Omega)\right]_{s-1}
=H_0^1(\Omega).
\end{align*}
Thus,
\begin{align}\label{danbian}
\left[H_0^1(\Omega),H^2(\Omega)\cap H^1_0(\Omega)\right]_{s-1}
\hookrightarrow
H^s(\Omega)\cap H_0^1(\Omega).
\end{align}
By Lemma \ref{dengtong}, we obtain
$\widetilde H^1(\Omega)=H^1_0(\Omega)$ and
$\widetilde H^2(\Omega)
\hookrightarrow
H^2(\Omega)\cap H^1_0(\Omega).$
This, together with
 Lemma \ref{lem:functoriality}, implies that
\begin{align}\label{xianzai}
\left[\widetilde H^1(\Omega),\widetilde H^2(\Omega)\right]_{s-1}
\hookrightarrow
\left[H_0^1(\Omega),H^2(\Omega)\cap H^1_0(\Omega)\right]_{s-1}.
\end{align}
Using Lemmas \ref{shichazhi} and
\ref{lem:Costabel}(i), we find that
\begin{align*}
\left[\widetilde H^1(\Omega),\widetilde H^2(\Omega)\right
]_{s-1}
=\widetilde H^s(\Omega)=H^s(\Omega)\cap H_0^1(\Omega).
\end{align*}
This, together with \eqref{xianzai}, implies that
\begin{align*}
H^{s}(\Omega)\cap H^1_0(\Omega)
\hookrightarrow
\left[H_0^1(\Omega),H^2(\Omega)\cap H^1_0(\Omega)\right]_{s-1},
\end{align*}
which, combined with \eqref{danbian}, further implies that \eqref{keypoint} holds. This proves (i).

Let $\Omega\subset\mathbb{R}^n$ be the bounded $C^1$ domain as in Lemma \ref{lem:Costabel}(ii) and let
$s\in (\frac{3}{2},2)$. Then, by Lemmas
\ref{lem:Costabel}(ii), \ref{dengtong}(ii), and \ref{shichazhi},
we find that
\begin{align*}
\left[H_0^1(\Omega),H^2(\Omega)\cap H^1_0(\Omega)\right]_{s-1}
&=\left[H_0^1(\Omega),H_0^2(\Omega)\right]_{s-1}=\left[\widetilde{H}^1(\Omega),\widetilde{H}^2(\Omega)\right]_{s-1}\\
&=\widetilde{H}^s(\Omega)
=H_0^s(\Omega)
=H^{s}(\Omega)\cap H^1_0(\Omega)
\end{align*}
with equivalent norms.
This shows (ii), which completes the proof of Theorem \ref{thm:sharp}.
\end{proof}

Let $\Omega\subset\mathbb{R}^n$ be a bounded Lipschitz domain.
For any $u,\phi\in H_0^1(\Omega),$ define
\begin{align*}
a(u,\phi):=\int_{\Omega}\nabla u(x) \cdot \nabla \phi(x) \,dx.
\end{align*}
As is well known, this form $a$ is bounded, positive, and coercive  on $H_0^1(\Omega).$
Denote the non-negative self-adjoint operator on $L^2(\Omega)$
associated with $a$ by $-\Delta_{D}$.  The operator $-\Delta_{D}$ is called the \emph{Dirichlet Laplacian} on
$L^2(\Omega).$
Denote the spectral decomposition of $-\Delta_{D}$ by $E(\lambda)$ (\cite[Chapter VIII]{rs72}).
Let $F:[0,\infty)\to \mathbb C$ be a measurable function.
Recall  that $F$ admits a
unique normal operator
$F(-\Delta_D)$ on $L^2(\Omega)$ with domain
\begin{align*}
D(F(-\Delta_D))
=
\left\{
f\in L^2(\Omega):
\int_0^\infty |F(\lambda)|^2\,d(E(\lambda)f,f)<\infty
\right\}.
\end{align*}
Moreover, for any $f\in D(F(-\Delta_D))$ and
$g\in L^2(\Omega),$
\begin{align*}
(F(-\Delta_D)f,g)
=
\int_0^\infty F(\lambda)\,d(E(\lambda)f,g)
\end{align*}
and
\begin{align}\label{jiyi}
\left\|F(-\Delta_D)f\right\|_{L^2(\Omega)}^2
=
\int_0^\infty |F(\lambda)|^2\,d(E(\lambda)f,f).
\end{align}
In particular, taking
$F(t):=t$, we obtain
\begin{align*}
D(-\Delta_D)=\left\{f\in L^2(\Omega):\ \int_0^\infty\lambda^2\,d(E(\lambda)f,f)<\infty\right\}.
\end{align*}

Now, we prove Theorem \ref{thm:Laplacian-application}.
\begin{proof}[Proof of Theorem \ref{thm:Laplacian-application}]
Let $t\in [0,\infty)$. Then, from the functional calculus and \eqref{jiyi}, we
deduce that
\begin{align*}
\left\|-\Delta_De^{t\Delta_D}v\right\|_{L^2(\Omega)}^2
=\int_0^\infty\lambda^2e^{-2t\lambda}\,d(E(\lambda)v,v)<\infty.
\end{align*}
Thus, $u(t)\in D(-\Delta_D)$.

Next, we prove
\begin{align}\label{pidou}
 u(t)\notin H^2(\Omega)\cap H_0^1(\Omega).
\end{align}
If $u(t)\in H^2(\Omega)\cap H_0^1(\Omega)$,
then, by Lemma \ref{lem:Costabel}(ii), we find that $u(t)\in H_0^2(\Omega)$.
From the definition of $H_0^2(\Omega)$, we infer that there exists a sequence $\{g_k\}_{k\in\mathbb{N}}\subset C^\infty_{\mathrm{c}}(\Omega)$ such that
\begin{align*}
\lim_{k\to\infty}g_k=u(t)\ \ \text{in}\ \ H^2(\Omega).
\end{align*}
Since $\{g_k\}_{k\in\mathbb{N}}\subset C^\infty_{\mathrm{c}}(\Omega)$, it follows that, for any $k\in\mathbb{N},$
\begin{align*}
\int_{\Omega}\Delta g_k(x)\,dx=0.
\end{align*}
Letting $k\to\infty$, we obtain
\begin{align}\label{renzhen}
\int_{\Omega}\Delta u(t)(x)\,dx=0.
\end{align}
Recall that $v$ is the weak solution of the Dirichlet problem \eqref{pos}.
Then
\begin{align*}
\Delta u(t)=\Delta_D e^{t\Delta_D}v= e^{t\Delta_D}\Delta_D v
=e^{t\Delta_D}1.
\end{align*}
By this and  the fact that
$-\Delta_D$ is non-negative and self-adjoint, we have
\begin{align*}
\int_{\Omega}\Delta u(t)(x)\,dx
=\int_{\Omega}e^{t\Delta_D}1(x)\,dx
=(e^{t\Delta_D}1,1)_{L^2(\Omega)}
=\left\|e^{t\Delta_D/2}1\right\|_{L^2(\Omega)}>0.
\end{align*}
This contradicts \eqref{renzhen}. Thus, \eqref{pidou} holds, which completes the proof of Theorem \ref{thm:Laplacian-application}.
\end{proof}

\section*{Declaration of competing interest}

The authors have no conflicts to disclose.

\section*{Data availability}

Data sharing is not applicable to this article as no new data were created or analyzed in this study.

\bigskip

\noindent
Xiaosheng Lin

\smallskip

\noindent
School of Mathematical Sciences, Jimei University,
Xiamen 361005, The People's Republic of China

\smallskip

\noindent {\it E-mail}: \texttt{xslin@jmu.edu.cn}

\bigskip

\noindent Dachun Yang (Corresponding author), Wen Yuan and Yangyang Zhang

\smallskip

\noindent Laboratory of Mathematics and Complex Systems
(Ministry of Education of China),
School of Mathematical Sciences, Institute for Advanced Study,
Beijing Normal University,
Beijing 100875, The People's Republic of China

\smallskip

\noindent{\it E-mails:} \texttt{dcyang@bnu.edu.cn} (D. Yang)

\noindent\phantom{{\it E-mails:}} \texttt{wenyuan@bnu.edu.cn} (W. Yuan)

\noindent\phantom{{\it E-mails:}} \texttt{yangyzhang@bnu.edu.cn} (Y. Zhang)

\bigskip

\noindent Sibei Yang

\medskip

\noindent School of Mathematics and Statistics, Lanzhou University, Lanzhou 730000, The People's Republic of China

\smallskip

\noindent{\it E-mail:} \texttt{yangsb@lzu.edu.cn}

\end{document}